\documentclass[english]{amsart}

\usepackage{amsmath}
\usepackage{amssymb}
\usepackage{amsthm}
\usepackage{amscd}
\usepackage{babel}
\usepackage[latin1]{inputenc}
\usepackage{graphicx}

\renewcommand{\subjclassname}{AMS \textup{2000} Mathematics Subject Classification:\ }

\newtheorem{cor}{Corollary}
\newtheorem{prop}{Proposition}

\newtheorem{con}{Question}
\newtheorem{conj}{Conjecture}
\theoremstyle{definition}
\newtheorem{defi}{Definition}
\newtheorem{rem}{Remark}

\author{J.M. Grau}
\address{Departamento de Matemáticas, Universidad de Oviedo\\ Avda. Calvo Sotelo, s/n, 33007 Oviedo, Spain}
\email{grau@uniovi.es}

\author{A. M. Oller-Marc\'{e}n}
\address{Centro Universitario de la Defensa\\ Ctra. de Heusca, s/n, 50090 Zaragoza, Spain}
\email{oller@unizar.es}

\author{M. Rodríguez}
\address{}
\email{rodlopmanuel@gmail.com}

\author{D. Sadornil}
\address{Departamento de Matemáticas, Estadística y Computación,
Universidad de Cantabria\\ F. Ciencias, Avda de los Castros s/n,
39005 Santander, Spain}
\email{sadornil@unican.es}

\title{Fermat test with gaussian base and Gaussian pseudoprimes}
\begin{document}
\maketitle

\begin{abstract}
The structure of the group $(\mathbb{Z}/n\mathbb{Z})^\star$ and Fermat's little theorem are the basis for some of best-known
primality testing algorithms. Many related concepts arise:
Euler's totient function and Carmichael's lambda function, Fermat
pseudoprimes, Carmichael and cyclic numbers, Lehmer's totient problem, Giuga's conjecture, etc. In this paper, we
present and study analogues to some of the previous concepts arising when we consider the underlying group
$\mathcal{G}_n:=\{a+bi\in\mathbb{Z}[i]/n\mathbb{Z}[i] :
a^2+b^2\equiv 1\ \textrm{$\pmod n$}\}$. In particular we characterize Gaussian Carmichael numbers via a
Korselt's criterion and we present their relation with Gaussian
cyclic numbers. Finally, we present the relation between Gaussian
Carmichael number and 1-Williams numbers for numbers $n \equiv 3
\pmod{4}$. There are also no known composite numbers less than
$10^{18}$ in this family that are both pseudoprime to base $1+2i$ and 2-pseudoprime.
\end{abstract}

\subjclassname{}

\section{Introduction}

Most of the classical primality tests are based on Fermat's little
theorem: let $p$ be  a prime number and let $a$ be an integer such
that $p\nmid a$, then $a^{p-1} \equiv  1 \pmod p$. This theorem gives
a possible way to detect non-primes: if for a certain $a$ coprime to
$n$, $a^{n-1}\not\equiv 1\pmod{n}$, then $n$ is not prime. The problem is
that the converse is false and there exists composite numbers
$n$ such that $a^{n-1}\equiv 1\pmod{n}$ for some
$a$ coprime to $n$. In this situation $n$ is called
pseudoprime with respect to base $a$ (or $a$-pseudoprime). A composite integer $n$ which is a pseudoprime to any base $a$ such that $\gcd(a,n)=1$ is called a Carmichael number (or absolut pseudoprime).

Fermat theorem can be deduced from
the fact that the units of $\mathbb{Z}/n\mathbb{Z}$ form a subgroup
of order $n-1$ when $n$ is prime. Associated to the subgroup
$(\mathbb{Z}/n\mathbb{Z})^\star$ we can define the well-know Euler's totient
function and Carmichael's lambda function which are defined in the
following way:
$$\varphi(n):=|(\mathbb{Z}/n\mathbb{Z})^\star|, \quad
\lambda(n):=\textrm{exp}(\mathbb{Z}/n\mathbb{Z})^\star.$$

It seems reasonable (and natural) to extend these ideas to other general groups $G_n$.
This extension leads to composite/primality tests according to the following steps:
\begin{itemize}
\item[1º)] Compute $f(n)=|G_n|$ under the assumption that $n$ is prime.
\item[2º)] Given $n$, if we can find $g\in G_n$ such that $|g|\nmid f(n)$, then $n$ is not prime.
\end{itemize}

This idea is present in tests based in lucasian sequences \cite{WI} and elliptic
curves \cite{SIL}. Recent works have developed these concepts in
other contexts. Pinch \cite{PINCH} considers primality tests
based on quadratic rings and discuss the absolute pseudoprimes for
them. Shettler \cite{JORDAN} studies analogues to
Lehmer's Problem Totient and Carmichael numbers in a PID. Steele \cite{ST} generalizes Carmichael numbers to number
rings introducing Carmichael ideals in number rings and proving an analogue to
Korselt's criterion for them.

Following these approaches, in this paper we consider the groups
$$\mathcal{G}_n:=\{a+bi\in\mathbb{Z}[i]/n\mathbb{Z}[i] :
a^2+b^2 \equiv 1\ \textrm{(mod $n$)}\}.$$

For these groups, we define the corresponding Euler and Carmichael
functions and we study some of their properties. We also present the
concepts of Gaussian pseudoprime and Gaussian Carmichael number
presenting an explicit Korselt's criterion. Cyclic numbers, Lehmer's
Totient Problem \cite{BCF} and Giuga's conjecture \cite{GIU} are
also considered in this gaussian setting.

It is known that Carmichael numbers have at least three prime
factors. We show that Gaussian Carmichael numbers with only two prime factors exist
and we determine their form. Moreover, although there are gaussian
pseudoprimes with respect to any base, if we combine our ideas with a
classical Fermat test, we show that no number of the form $4k+3$ smaller that than $10^{18}$ 
passes both tests (for some particular
bases). This strength is possible due
to a relationship with 1-Williams numbers \cite{WI} that we make explicit.

\section{Preliminaries}
In this section we determine the order and structure of the group $\mathcal{G}_n$. We also show some
elementary properties and relations between the Gaussian counterparts of Euler and Carmichael functions.

For any positive integer $n$ we will denote by $\mathcal{I}_n$ the ring of gaussian integers modulo $n$; i.e.,
$$\mathcal{I}_n:=\{a+bi : a,b\in\mathbb{Z}/n\mathbb{Z}\}=\mathbb{Z}[i]/n\mathbb{Z}[i].$$
Further, we will consider the group $\mathcal{G}_n$ defined by
$$\mathcal{G}_n:=\{a+bi\in\mathcal{I}_n : a^2+b^2\equiv 1\ \textrm{(mod $n$)}\}.$$
Once we have defined the group
we can define the following arithmetic functions:
$$\Phi(n):=|\mathcal{G}_n|,\ \ \rightthreetimes(n):=\textrm{exp}(\mathcal{G}_n).$$
Note that $\Phi$ and $\rightthreetimes$ are the analogues to Euler's totient funtion and Carmichael's lambda functions, respectively.

It is quite clear that if $n=p_1^{r_1}\cdots p_s^{r_s}$, then
$$\mathcal{G}_n\cong \mathcal{G}_{p_1^{r_1}}\times\cdots\times\mathcal{G}_{p_s^{r_s}}.$$
As a consequence, if $\gcd(m,n)=1$, $\Phi(mn)=\Phi(m)\Phi(n)$ and $\rightthreetimes(mn)=\textrm{lcm}(\rightthreetimes(m),\rightthreetimes(n))$. Hence, in order to study the group 
$\mathcal{G}_n$ we can restrict ourselves to the case when $n$ is a
prime power. 

\begin{prop}
Let $p$ be a prime and let $k>0$ be an integer. Then:
$$\mathcal{G}_{p^k}\cong\begin{cases} C_2, & \textrm{if $p=2$ and $k=1$};\\   C_{2^{k-2}}\times C_2\times C_4, & \textrm{if $p=2$ and $k\geq 2$};\\ C_{p^{k-1}}\times C_{p-1}, & \textrm{if $p\equiv 1$ (mod $4$)};\\ C_{p^{k-1}}\times C_{p+1}, & \textrm{if $p\equiv 3$ (mod $4$)}.\end{cases}$$
\end{prop}
\begin{proof}
We will focus only on the case $p\equiv 3 \pmod 4$. In this case, it
is well-known  that $\mathcal{G}_p\cong GF(p^2)^{\star}$. Since
$\mathcal{G}_p$ is a subgroup of $GF(p^2)^{\star}$, it must be
cyclic. Moreover, counting quadratic residues it can be seen that
$|\mathcal{G}_p|=p+1$ and, consequently, $\mathcal{G}_p\cong
C_{p+1}$.

We can now apply the Fundamental Lemma in \cite[p. 587]{GOL} to
obtain that $|\mathcal{G}_{p^{k}}|=p^{k-1}(p+1)$. This means that,
if $\Phi:\mathcal{G}_{p^k}\rightarrow \mathcal{G}_p$ is the
$\pmod{p}$ group homomorphism, then $|\textrm{Ker}\ \Phi|=p^{k-1}$.
Finally, observe that $\textrm{Ker}\ \Phi$ is an abelian $p$-group
with exactly $p-1$ elements of order $p$, namely
$\{1+Bp^{k-1}i\in\mathcal{G}_{p^{k}} : 1\leq B\leq p-1\}$.
Consequently it must be cyclic and the proof is complete in this
case.
\end{proof}

As a straightforward consequence we compute $\Phi(p^k)$ and
$\rightthreetimes(p^k)$.

\begin{cor}\label{corla2}
Let $p$ be a prime and let $k>0$ be an integer. Then
$$\Phi(p^k)=\begin{cases} 2, & \textrm{if $p=2$ and $k=1$};\\ 2^{k+1}, & \textrm{if $p=2$ and $k>1$};\\ p^{k-1}(p-1), & \textrm{if $p\equiv 1$ (mod $4$)};\\ p^{k-1}(p+1), & \textrm{if $p\equiv 3$ (mod $4$)}.\end{cases}$$
$$\rightthreetimes(p^k)=\begin{cases} 2, & \textrm{if $p=2$ and $k=1$};\\ 4, & \textrm{if $p=2$ and $k=2,3,4$};\\ 2^{k-2}, & \textrm{if $p=2$ and $k\geq 5$}; \\ p^{k-1}(p-1), & \textrm{if $p\equiv 1$ (mod $4$)};\\ p^{k-1}(p+1), & \textrm{if $p\equiv 3$ (mod $4$)}.\end{cases}$$
\end{cor}

For an odd prime number $p$, let us define
$\beta(p)=\left(\frac{-1}{p}\right)$ and put $\beta(2)=0$. With
this notation the following result is straightforward.

\begin{prop}\label{Phin}
$$\Phi(n)=\begin{cases} \displaystyle{2n\prod_{p|n}\left(1-\frac{\beta(p)}{p}\right)}, &
\textrm{if $4$ divides $n$}; \\
\displaystyle{n\prod_{p|n}\left(1-\frac{\beta(p)}{p}\right)}, &
\textrm{otherwise}.\end{cases}$$
\end{prop}

Recall that $\Phi(mn)=\Phi(m)\Phi(n)$ provided $\gcd(m,n)=1$. The following result describes the general situation.

\begin{prop}\label{prod}
Let $m,n\in\mathbb{N}$. Then

$$\Phi(nm)=\Phi(n)\Phi(m)\frac{\gcd(m,n)}{\Phi(\gcd(m,n))}.$$
\end{prop}
\begin{proof}
It is enough to consider the prime power decomposition of $m$ and
$n$.
\end{proof}

In particular, if we put $m=n$ we obtain the following.

\begin{cor}
Let $n,s\in\mathbb{N}$. Then
$$\Phi(n^m)=\begin{cases} n^{m-1}\Phi(n), & \textrm{if $n\not\equiv 2$ (mod $4$)};\\ 2n^{m-1}\Phi(n) & \textrm{if $n\equiv 2$ (mod $4$)}.\end{cases}$$
\end{cor}

\begin{prop}
Let $m,n\in\mathbb{N}$. If $d=\textrm{gcd}(m,n)$ and $M= \textrm{lcm}(m,n)$, then:
$$\Phi(d)\Phi(M)=\Phi(m)\Phi(n).$$
\end{prop}
\begin{proof}
Recall that $M=n\frac{m}{d}$, where $\gcd(n,\frac{m}{d})=1$ and we can assume, without loss of generality, that $\gcd(m,d)=1$. Then,
Proposition \ref{prod} leads to:
$$\Phi(M)=\Phi\left(n\frac{m}{d}\right)=\Phi(n)\Phi\left(\frac{m}{d}\right))\Phi(n)\frac{\Phi(m)}{\Phi(d)}$$
and the result follows.
\end{proof}

Recall that for the classical Euler and Carmichael functions, $\phi(n)=\lambda(n)$ if and only if $n=2$, $n=4$ or $n=p^r,2p^r$ for some odd prime $p$ and $r>0$. Note that in all these cases the group
$(\mathbb{Z}/n\mathbb{Z})^\star$ is cyclic. For our recently defined
functions $\Phi$ and $ \rightthreetimes$ we have the following:

\begin{prop} 
$\Phi(n)=\rightthreetimes(n)$ if and only if $n=2$ or $n=p^r$ for some odd prime $p$ and $r>0$.
\end{prop}
\begin{proof} 
Just apply Corollary \ref{corla2} and recall that if $\gcd(m,n)=1$, then $\Phi(mn)=\Phi(m)\Phi(n)$ while
$\rightthreetimes(mn)=\textrm{lcm}(\rightthreetimes(m),\rightthreetimes(n))$.
\end{proof}

We end this section showing that the asymptotic behavior of
$\Phi(n)$ is not exactly the same as that of his classical counterpart.

\begin{prop}
$$\lim \inf \frac{\phi(n)}{n}=\lim \inf \frac{\Phi(n)}{n}=0$$
$$1=\lim \sup \frac{\phi(n)}{n} \neq \lim \sup \frac{\Phi(n)}{n}=\infty$$
\end{prop}
\begin{proof}
For the asymptotic growth of Euler $\phi$ function and its limits
see \cite{HW}.

Now consider sequences $\{S_n\}$ and $\{L_n\}$ given by:
$$
S_n:= \prod_{\substack{p\leq n\\ p \equiv 3 \pmod 4}}p,  \qquad
L_n:= \prod_{\substack{p\leq n\\ p \equiv 1 \pmod 4}}p.
$$

We have that $\Phi(p)=p+1$ for every odd prime $p \equiv 3 \pmod 4$,
hence
$$
\lim_{n\rightarrow\infty}\frac{\Phi(S_n)}{S_n}=\lim_{n\rightarrow\infty}
\prod_{\substack{p\leq n\\ p \equiv 3 \pmod 4}}
\frac{p+1}{p}=\infty,
$$
since $\prod \frac{p+1}{p}\geq 1+\sum 1/p$ and this series is
divergent by the strong form of Dirichlet's theorem. On the other
hand,
$$
\lim_{n\rightarrow\infty}\frac{\Phi(L_n)}{L_n}=
\lim_{n\rightarrow\infty}\prod_{\substack{p\leq n\\ p \equiv 1 \pmod 4}}\bigl(1-\frac{1}{p}\bigr).$$

Moreover, $$0\leq \prod_{\substack{p\leq n\\ p \equiv 1 \pmod 4}}\bigl(1-\frac{1}{p}\bigr) \leq \prod_{\substack{p\leq n\\ p \equiv 1 \pmod
4}} e^{-1/p}= e^{-\sum \frac{1}{p}},$$

where the sum in the exponent is taken over the primes $p\equiv 1 \pmod 4,\,  p\leq
n$. Again, by the strong form of Dirichlet's theorem, this fuction
tends to 0 and result holds.
\end{proof}

\section{Gaussian Fermat pseudoprimes}

We start this section introducing the arithmetic function
$\mathcal{F}$, which will play the same role as $n-1$ plays in the classical
setting.

$$
\mathcal{F}(n)=\begin{cases} n-1, & \textrm{if $n\equiv 1$ (mod $4$)};\\
n+1, & \textrm{if $n\equiv 3$ (mod $4$)};\\
n, & \textrm{otherwise}.\end{cases}
$$

Note that, if $n$ is prime, $\mathcal{F}(n)=|\mathcal{G}_n|$.

We present the analogue to Fermat's little theorem in this gaussian setting.

\begin{prop}\label{FermatGaussiano}
Let $p$ be a prime number and let $z$ be a gaussian integer such
that $p$ is coprime with $z \overline{z}$. Then:
\begin{itemize}
\item[i)]
$\left(\displaystyle{\frac{z}{\overline{z}}}\right)^{\mathcal{F}(p)}\equiv
1 \pmod{p}.$
\item[ii)] $Im(z^{\mathcal{F}(p)})\equiv 0 \pmod{p}.$
\end{itemize}
\end{prop}
\begin{proof}
Note that if $z\in\mathbb{Z}[i]$ is such that $\gcd(n,z\overline{z})=1$, then $z/\overline{z}\in\mathcal{G}_n$. Hence, it is enough to apply Corollary \ref{corla2}.
\end{proof}

\begin{rem}
Both conditions in Proposition \ref{FermatGaussiano} are equivalent.
\end{rem}

We can consider the above result as a
compositeness test for integers: if for some integer $n$ we find
a gaussian integer $z$ such that either condition i) or ii) does not hold,
then $n$ is a composite number. Nevertheless, like in the classical setting, the converse is not always true. This fact motivates
the following definition:

\begin{defi}
A composite integer $n$ is called a Gaussian Fermat pseudoprime
(GFP) with respect to the base $z \in \mathbb{Z}[i]$ if
$\gcd(n,z\overline{z})=1$ and condition i) (or equivalently ii)) from
Proposition \ref{FermatGaussiano} holds for $n$.
\end{defi}

In the classical setting the choice of different basis leads, in general, to different sets of associated Fermat pseudoprimes. In our case it is easy to describe a family of different basis leading to the same set of associated Gaussian Fermat pseudoprimes.

\begin{prop}
Let $z,w$ be two gausian integers such that  $|z|=|w|$. Then an integer $n$ is a Gaussian
Fermat pseudoprime with respect to $z$ if and only if $n$ is a Gaussian
Fermat pseudoprimes with respect to $w$.
\end{prop}
\begin{proof}
Assume that $n$ is a GFP with respect to $z$. Then $\gcd(n,z\overline{z})=1$ and $(z/\overline{z})^{\mathcal{F}(n)}\equiv 1\pmod{n}$. Now, since $|w|=|z|$ we have that $\gcd(n,w\overline{w})=\gcd(n,z\overline{z})=1$. Moreover, since $(z/\overline{z})^{\mathcal{F}(n)}\equiv 1\pmod{n}$ and $z/\overline{z}\in\mathcal{G}_n$ it follows that $\rightthreetimes(n)\mid \mathcal{F}(n)$. Hence, $(w/\overline{w})^{\mathcal{F}(n)}\equiv 1\pmod{n}$ because $w/\overline{w}\in\mathcal{G}_n$. The converse is clear since the roles of $z$ and $w$ are symmetric and the proof is complete.
\end{proof}

\section{Gaussian Carmichael and cyclic numbers}

An integer $n$ that is a Fermat pseudoprime for all bases  coprime
to $n$ is called a Carmichael number \cite{CARMI}. In the gaussian
case there also exists composite numbers which are GFP with respect
all bases.

\begin{defi}
A composite number $n\in\mathbb{N}$ is a Gaussian Carmichael number
($G-$Carmichael)  if it is a GFP to base $z$ for every gaussian
integer $z$ such that $n$ is coprime to $z \overline{z}$.
\end{defi}

An alternative and equivalent definition of Carmichael numbers is
given by Korselt's criterion \cite{KORSELT} which states that a
positive composite integer $n$ is a Carmichael number if and only if
$n$ is square-free, and for every prime divisor $p$ of $n$, $p-1$
divides $n-1$. It follows from this characterization that all Carmichael numbers are
odd. A similar  characterization of $G-$Carmichael numbers can be given,
showing that there are also even $G-$Carmichael numbers.

\begin{prop}\label{car2}
For every composite integer $n$ the following are equivalent.
\begin{itemize}
\item[i)] $n$ is $G-$Carmichael number.
\item[ii)] $\rightthreetimes(n)$ divides $\mathcal{F}(n)$.
\item[iii)] For every prime divisor $p$ of $n$, $\mathcal{F}(p)$ divides $\mathcal{F}(n)$ and one of the following conditions holds:
\begin{itemize}
\item[a)] $n$ is odd and square-free,
\item[b)] $n$ is multiple of $4$ and  $\frac{n}{4}=2,3,5$ or not a prime number.
\end{itemize}
\end{itemize}
\end{prop}

\begin{proof}
Since $\rightthreetimes(n)$ is the exponent of the group
$\mathcal{G}_n$, i) and ii) are clearly equivalent.

From Corollary \ref{corla2} and the fact that
$\rightthreetimes(mn)=\textrm{lcm}(\rightthreetimes(m),\rightthreetimes(n))$
if $\gcd(m,n)=1$, it is easy to see that iii) implies ii) when $n$
is a number that satisfies a) or b).

Finally, assume now that $\rightthreetimes(n)$ divides
$\mathcal{F}(n)$ and let be $n=2^ap_1^{r_1}\cdots p_s^{r_s}$. We
have that
$\rightthreetimes(n)=\textrm{lcm}(\rightthreetimes(2^a),\rightthreetimes(p_1^{r_1}),\dots,\rightthreetimes(p_s^{r_s}))$,
so $\rightthreetimes(2^a)$ and $\rightthreetimes(p_i^{r_i})$ divides
$\mathcal{F}(n)$. From Corollary \ref{corla2}, it is clear that for
every prime $p$, $\rightthreetimes(p)=\mathcal{F}(p) $ divides
$\rightthreetimes(p^k)$ with $k\geq 1$ and $\mathcal{F}(p)$ also
divides $\mathcal{F}(n)$ as claimed.

If $n$ is odd ($a=0$) and $r_i\geq 2$ for some $i\in\{1,\dots,s\}$
we get that $p_i$ divides $\rightthreetimes(n)$ and, consequently,
also $\mathcal{F}(n)$. Thus, $p_i$ divides $n-1$ or $n+1$ which is a
contradiction and $n$ must be square-free in this case.

We now turn to the even case. If $a=1$ and $n$ is divisible by an
other prime $p$ such that $p\equiv 1 \pmod{4}$, then $p-1$ divides
$\mathcal{F}(n)=n$. Hence $n$ is a multiple of $4$, a contradiction.
The same follows if there exist a prime $p\equiv 3 \pmod{4}$
dividing $n$ so we conclude that if $n\neq 2$ is even, it must be a
multiple of $4$.

Now, let be $n=4p$ with $p$ a prime. If $p=2$, then $n=8$ and we are
done. If $p\equiv 1 \pmod{4}$, it follows that $p-1$ divides $n$;
i.e., $p-1$ divides 4 so $p=5$ and $n=20$. Finally, if $p\equiv 3
\pmod{4}$, it follows that $p+1$ divides 4 so $p=3$ and $n=12$.
Hence we have seen that if $4$ divides $n$ and $n\neq 8,12,20$, then
$\frac{n}{4}$ is not prime and the proof in complete.
\end{proof}

In 1994 it was shown by Alford, Granville y Pomerance \cite{ALF}
that there exist infinitely many Carmichael numbers. It is
easy to see that every power of 2  is a $G-$Carmichael number, hence there
are also infinitely many of them. However, if we restrict to odd $G-$Carmichael
numbers, the problem  seems to have at least the same difficulty
as the classical case. 

Carmichael numbers have at
least three prime factors. We know that 12 and 20 are only even $G-$Carmichael numbers with only two prime factors. The following result describes the family of odd $G-$Carmichael numbers with exactly two prime factors.

\begin{prop}\label{Prop2p}
Let $p<q$ be odd primes. Then $n=pq$ is a  gaussian Carmichael
number if and only if $p$ and $q$ are twin primes such that $8$
divides $p+q$.
\end{prop}
\begin{proof}
Assume that $n=pq$ with $p<q$ odd primes is a $G-$Carmichael number.

If $p,q\equiv 1 \pmod 4$ , $\mathcal{F}(n)=n-1=pq-1$. From
proposition \ref{car2},  $p-1$ divides $pq-1=(p-1)(q+1)+q-p$. Hence
$p-1$ divides $q-p$ and also $q-1=(q-p)+(p-1)$. In the same way
$q-1$ divides $p-1$. So $p-1=q-1$ which is impossible. If $p,q\equiv 3
\pmod {4}$ we reach a similar contradiction using the same ideas. If
$p\equiv 1 \pmod{4}$ and $q\equiv 3 \pmod{4}$ we obtain that $p=q+2$
which is impossible because $p<q$.

Thus  $p\equiv 3  \pmod{4}$ and $q\equiv 1 \pmod{4}$ and  we have,
whit similar reasoning that $q=p+2$. Moreover, $p+q=2p+2\equiv 0
\pmod{8}$.

The converse is trivially true and the proof is complete.
\end{proof}

Recall that a positive integer $n$ which is coprime to $\phi(n)$ is called a cyclic number (sequence A003277 in \cite{OEIS}). This
terminology comes from group theory since a number $n$ is cyclic if and only if any group of order $n$
is cyclic \cite{SZ}. From  Korselt's criterion it follows that any divisor of a Carmichael number is cyclic. In the gaussian setting we define Gaussian cyclic numbers in the following way.

\begin{defi}
An integer $n$ is called $G-$cyclic if $\gcd(\Phi(n),n)=1$.
\end{defi}

The relationship between G-Carmichael and G-cyclic numbers is the same as
in the setting, the proof being also quite similar.

\begin{prop}
Any divisor of a odd $G-$Carmichael number is $G-$cyclic.
\end{prop}
\begin{proof}
Let $n$ be an odd $G-$Carmichael number. Since $n$ is square-free,
$n=p_1p_2\cdots p_r$  and from proposition \ref{Phin},
$\Phi(n)=\prod (p_i-\beta(p_i))$. A divisor $d$ of $n$ is a product
of some of these primes, that is, $d=\prod_{h \in J} p_h$,
$J\subseteq \{1,2,\ldots, r\}$. If $GCD(\Phi(d),d)<>1$,, then there
exist two indices $i\neq k$ in $J$ such that $p_i$ divides
$p_k-\beta(p_k)$. As $n$ is a Carmichael number, we also have
$p_k-\beta(p_k)$ divides $n-\beta(n)$. Hence, $p_i$ divides
$n-\beta(n)$ which is absurd since $n$ is divisible by $p_i$ and
$\beta(p_i)=\pm 1$.
\end{proof}

Around 1980, G. Michon conjectured that all odd cyclic numbers have Carmichael
multiples. This can be reasonably extended to $G-$cyclic numbers
and we can ask if all odd $G-$cyclic numbers have $G-$Carmichael
multiples.

Cyclic numbers can also be characterized in terms on congruences. A
number $n$ is cyclic if and only if it satisfies $\phi(n)^{\phi(n)}
\equiv 1 \pmod n$ or $\lambda(n)^{\lambda(n)} \equiv 1 \pmod
n$. In our situation only one implication remains valid, namely.

\begin{prop}
If  $\Phi(n)^{\Phi(n)} \equiv  1 \pmod n$ or
$\rightthreetimes(n)^{\rightthreetimes(n)} \equiv  1 \pmod n$,
then $n$ is a $G-$cyclic number.
\end{prop}

\begin{proof}
Let $n$ be a positive integer such that $\Phi(n)^{\Phi(n)} \equiv  1
\pmod n$. Then, for any prime divisor $p$ of $n$ it holds that
$\Phi(n)^{\Phi(n)} \equiv  1 \pmod p$. Now, if $n$ is not a $G-$cyclic
number, there exists a prime $p$ with $p \mid \gcd(\Phi(n),n)$.
Thus, $p$ divides $\Phi(n)$ and $\Phi(n)^{\Phi(n)} \equiv  0 \pmod p$,
a contradiction.

On the other hand, let $n=p_1^{e_1}p_2^{e_2}\cdots p_r^{e_r}$ be a
positive integer. If $n$ is not a $G-$cyclic number then there exists
a prime $p$ which $p \mid \gcd(\Phi(n),n)$. Since $p\mid
\Phi(n)=\Phi(p_1^{e_1})\Phi(p_2^{e_2})\cdots \Phi(p_r^{e_r})$, there
exists $1\leq j\leq r$ with $p\mid \Phi(p_j^{e_j})$. If $p_j=2$ then
$p=2$ and $\rightthreetimes(n)$ is even. Otherwise,
$\Phi(p_j^{e_j})=\rightthreetimes(p_j^{e_j})$ and $p$ divides
$\rightthreetimes(n)$. So $\rightthreetimes(n)^{\rightthreetimes(n)}
\equiv 0 \pmod p$ and $n$ do not satisfy the hypothesis. 
\end{proof}

The converse of the previous proposition is no true. In fact there are $G-$cyclic numbers $n$ that do not satisfy any of
the above conditions. The first of them being:

$$ 77, 119, 133, 187, 217, 253, 287, 301, 319, 323, 341, 391,\ldots $$

\section{$G-$Lehmer's totient problem and $G-$Giuga's conjecture}

Lehmer's totient problem, named after D. H. Lehmer, asks whether there
is any composite number $n$ such that $\phi(n)$ divides $n - 1$. This is true for every prime number, and
Lehmer conjectured in 1932 \cite{LE} that the answer to his question was negative. He showed that if any such $n$ exists, it must be odd,
square-free, and divisible by at least seven primes. This numbers,
called Lehmer numbers,  are clearly Carmichael numbers and, up to date, none has been found. It is known that these numbers
have at least 15 prime factors and are greater than $10^{30}$. Moreover, if a Lehmer number is divisible by 3, the number of
prime factors increases to $40000000$ with more than 360000000
digits (see \cite{BCF}). We now define our analogue concept.

\begin{defi}
A composite number $n$ is a $G-$Lehmer number if $ \Phi (n) \mid \mathcal{F}(n)$.
\end{defi}

It is clear that every $G-$Lehmer number is a $G-$Carmichal number.
Besides, it is easy to note that $G-$Lehmer numbers exist.

\begin{prop}
Let $p<q$ be odd primes. Then $n=pq$ is a $G-$Lehmer number if and
only if $p$ and $q$ are twin primes such that $8$ divides $p+q$.
\end{prop}
\begin{proof}
As $n$ must be a $G-$Carmichael number, $n=(4k-1)(4k+1)$ where $4k\pm
1$ are both primes. From this numbers
$\Phi(n)=\mathcal{F}(n)=(4k)^2$, so $n$ is a $G-$Lehmer number.
\end{proof}

Note that, from Proposition \ref{Prop2p} this result means that every odd $G-$Carmichael number with exactly 2 prime factors is a $G-$Lehmer number. Nevertheless, there are $G-$Lehmer numbers with more thatn 2 prime factors (A182221  in \cite{OEIS}):

$$ 255, 385, 34561, 65535, 147455, 195841, \dots  $$

This suggests an interesting question:

\begin{con}
Are there infinitely many $G-$Lehmer numbers?
\end{con}

Furthermore, all known $G-$Lehmer numbers satisfy that $ \mathcal{F}(n)
=\Phi(n) $. Hence, it is reasonable to propose the $G-$Lehmer's
Totient problem:

\begin{con}
Is there any number $n$ such that $\phi(n)$ is a proper divisor of
$\mathcal{F}(n)$?
\end{con}

In 1932, Giuga \cite{GIU} proposed another conjecture about prime
numbers. He postulated that a number $p$ is prime if and only if $
\sum i^{p-1}\equiv -1 \pmod p $, where the sum is taken over all
integers $1\leq i\leq p-1$. Giuga showed that there are no
exceptions to his conjecture up to $10^{1000}$. This was later
improved to $10^{13800}$ \cite{BOR}. A similar approach to
Giugas's conjecture, replacing $n-1$ by $\mathcal{F}(n)$, leads us
to consider the following set, which contains all prime numbers.

$$\mathfrak{G} := \{n \in \mathbb{N}:\sum_{z \in \mathcal{G}_n}
z^{\mathcal{F}(n)} \equiv \mathcal{F}(n) \pmod{n}\}$$

However, this set also contains lots of composite numbers. For
example, every power of 2 is in $\mathfrak{G}$. For odd integers we
have the next result.

\begin{prop}
Let be $n$ an odd integer. If $\Phi(n)=\mathcal{F}(n)$, then $n \in
\mathfrak{G}$.
\end{prop}
\begin{proof}
Since $|\mathcal{G}_n|=\Phi(n)$, for all $z \in \mathcal{G}_n$,
$z^{\Phi(n)} \equiv 1 \pmod n$. If $\Phi(n)=\mathcal{F}(n)$ then
$$\sum_{z \in \mathcal{G}_n} z^{\Phi(n)} \equiv |\mathcal{G}_n|
\equiv  \Phi(n) \equiv \mathcal{F}(n) \pmod n,$$ and $n$ is in
$\mathfrak{G}$. \end{proof}

Thus, prime numbers and every known $G-$Lehmer numbers are in
$\mathfrak{G}$. Furthermore, no other odd composite integer is known
to be in $\mathfrak{G}$. So, we formulate the following conjecture
regarding numbers in $\mathfrak{G}$.

\begin{conj}
For every odd $n$,  $n \in \mathfrak{G}$ if and only if
$\Phi(n)=\mathcal{F}(n)$.
\end{conj}

\section{Gaussian Fermat test for numbers of the form $4k+3$. }

The use of gaussian integers to perform the equivalent of Fermat's
little theorem to test primality is not just a mere theoretical
speculation. Lucas pseudoprimes \cite{WI} for some particular
sequences can be also seen as gaussian pseudoprimes. However,
Gaussian integers, and the corresponding definition of peudoprimes
using powers, is more similar to the classical one than the concept
of Lucas sequences.

As we have said before, we can take advantage of Proposition
\ref{FermatGaussiano} to test primality (more precisely
compositeness) of a number. This is that we call the Gaussian Fermat
Test with respect to the base $z$. Computational evidence reveals
that this test, based on the structure of $\mathcal{G}_N$, is very
powerful when it is combined with the classical one; i.e., there
are very few common pseudoprimes. Furthermore, this combination is
more stronger if we restrict to  numbers of the form $4k +3 $. From the William Galway list
\cite{GAL}, we
have checked that every Fermat pseudoprime number to base 2 less
than $10^{18}$ and of the form $4k +3 $ is not a Gaussian pseudoprime to base $z=1+2i$.

Baillie-PSW primality test \cite{PSW}, used in a lot of computer
algebra systems and software packages, is also a combination of two
primality tests. More precisely it is a strong Fermat probable prime
test to base 2 and a strong Lucas probable prime test. As the
previous combination, no composite number below $10^{19}$ passes it,
but it considers two strong type-test in contrast of our two basic
Fermat type tests. On the other hand, there are integers of the form $4k+3$ which are both Fermat pseudoprimes
to base 2 and Lucas pseudoprimes (see sequence
A227905 in \cite{OEIS}).

In general, combinations of two Fermat test with respect to two
different prime basis (less than 30) present more than 10 (and a mean of
34) pseudoprimes lower than $ 4 \cdot 10^7 $ of the form $ 4k +3$.
Even if we combine two basis to test if a number $n$ is a prime
using the Gaussian Fermat Test, there are more  pseudoprimes.
However, there is no composite number of the form $4k+3$ less than
$4 \cdot 10^7$ which is both a Gaussian pseudoprimes with respect to $1+2i$
and a Fermat pseudoprime with respect to a prime base less than 30.
The lowest base to be used to find a Fermat pseudoprime with respect
this base which is also a Gaussian Fermat pseudoprime to the base $
1 +2 i $ is $10$. 
Also with other Gaussian basis the combination with a Fermat test is
very strong as it is shown in the following table, which presents
the number of composite integers less than $ 4 \cdot 10^7 $ which
are simultaneously Gaussian Fermat pseudoprimes with respect to a
base $z$ (horizontal) and Fermat pseudoprimes with respect to a base
$a$ (vertical).

$$\begin{tabular}{|c |c |c | c|c | c|c |c | c|c |c |}
  \hline
   base & 2 & 3 & 4 & 5 & 6 &7 & 8 & 9 & 10 & 11  \\ \hline
  1 + 2$i$ & 0 & 0 & 0 & 0 & 0 & 0 & 0 & 0 & 1 & 0 \\ \hline
  1 + 4$i$ &  0 & 0 & 1 & 0 & 0 & 0 & 0 & 0 & 0 & 1  \\ \hline
  1 + 6$i$ & 0 & 1 & 2 & 0 & 2 & 0 & 0 & 1 & 0 & 1 \\ \hline
 1 + 10$i$ &  0 & 1 & 1 & 0 & 0 & 0 & 2 & 1 & 2 & 1 \\ \hline
  2 + 5$i$ & 0 & 0 & 1 & 0 & 1 & 0 & 0 & 0 & 0 & 1 \\ \hline
 2 + 7$i$ & 0 & 0 & 1 & 0 & 1 & 0 & 2 & 1 & 0 & 1 \\ \hline
 3 + 8$i$ & 0 & 1 & 2 & 1 & 0 & 0 & 1 & 1 & 0 & 1 \\ \hline
  3 + 10$i$ &  0 & 1 & 2 & 0 & 1 & 0 & 2 & 1 & 1 & 1  \\ \hline
  4 + 5$i$ & 0 & 0 & 1 & 0 & 0 & 0 & 1 &    0 & 0 & 1 \\ \hline
  4 + 9$i$ &   0 & 0 & 1 & 0 & 0 & 0 & 1 & 0 & 0 & 1  \\ \hline
\end{tabular} $$

One of the reasons explaining this phenomenon is that
Carmichael numbers, which always appear when combining two
classical Fermat tests, are avoided when we combine a Fermat test and a
Gaussian Fermat test, because Carmichael numbers are not necessarily
$G-$Carmichael numbers and conversely. In fact, there are no Carmichael
numbers of the form $4k +3 $ smaller than  $ 10 ^{18} $ which are
also $G-$Carmichael numbers.

Recall that an integer $n$ is
called an $r-$Williams number \cite{Ec,WI} if 
$$p \mid n\Rightarrow
(p+r) | (n+r)\textrm{ and } (p-r) | (n-r)$$

The following result relates our previous discussion with $1-$Williams numbers.

\begin{prop}
An odd number $n \equiv 3 \pmod 4$ is simultaneously a Carmichael
number and a $G-$Carmichael number if and only if $n$ is an $1-$Williams
number and $p \equiv 3 \pmod 4$ for every $p$ dividing $n$.
\end{prop}
\begin{proof}
Let $n \equiv 3 \pmod 4$, then
$\mathcal{F}(n)=n+1$. If $n$ is both a Carmichael and a $G-$Carmichael number we have that, for every $p$ dividing $n$:
$$
  \begin{array}{ll}    p-1 \mid n-1, &  \\
    p-1\mid n+1, & \hbox{if $p\equiv 1 \pmod{4}$,} \\
    p+1\mid n+1, & \hbox{if $p\equiv 3 \pmod{4}$.}
  \end{array}
$$

Now, if there exists a prime factor $p\equiv 1 \pmod{4}$ it follows that
$n-1=(p-1)k\equiv 0\pmod{4}$, a contradiction. Hence, every prime
factor is congruent with 3 modulo 4 and $n$ is a $1-$Williams
number.

On the other hand, if $n$ is a $1-$Williams number, then for each prime
factor $p$ of $n$ we have $p-1\mid n-1$ and $p+1\mid n+1$, so $n$ is
a Carmichael number. If $n$ is to be a $G-$Carmichael it is also necessary that
every factor $p\equiv 1 \pmod{4}$ satisfies $p-1\mid n+1$. But, by
hypothesis, $n$ does not have this kind of factors and the result follows. 
\end{proof}

Thus, the search for a number of the form $n \equiv 3 \pmod 4$ which is both a $G-$Carmichael number and a Carmichael
number is harder than to find a $1-$Williams number and, up to date, no $1-$Williams number is known

\end{document}